\documentclass[a4paper,12pt]{article}
\usepackage{amsthm}
\usepackage{mathrsfs}

\usepackage{amssymb}
\usepackage{epsfig}
\usepackage{amsfonts}
\usepackage{amsmath}
\usepackage{euscript}
\usepackage{amscd}
\usepackage{amsthm}
\usepackage{theoremref}
\DeclareMathAlphabet{\mathpzc}{OT1}{pzc}{m}{it}

\newtheorem{thm}{Theorem}[section]
\newtheorem{lem}[thm]{Lemma}
\newtheorem{prop}[thm]{Proposition} 
\newtheorem{cor}[thm]{Corollary}
\newtheorem{rem}[thm]{Remark}

\newtheorem{definition}{Definition}[section]

\newcommand{\p}{\mathpzc{p}}
\newcommand{\q}{\mathpzc{q}}

\newcommand{\hgt}{\operatorname{ht}}

\newcommand{\LND}{\operatorname{LND}}

\newcounter{para}

\title{A note on homogeneous rank $2$ locally nilpotent derivations on $k[X,Y,Z]$}
\author{Parnashree Ghosh\\
 {\small{\it Theoretical Statistics and Mathematics  Unit, Indian Statistical Institute,}}\\ 
  {\small{\it 203 B.T.Road, Kolkata-700108, India}}\\
   {\small{\it e-mail : parnashree$\_$r@isical.ac.in, ghoshparnashree@gmail.com}}
 }

\begin{document}
\date{}
\maketitle

\abstract{In this article we show that for every prime number $p$, any irreducible homogeneous locally nilpotent derivations of rank $2$ and degree $p-2$ are triangularizable. Further, we describe the structure of irreducible non-triangularizable homogeneous locally nilpotent derivations of rank $2$ and degree $pq-2$, where $p,q$ are prime numbers. Consequently, we give explicit descriptions of the generators of the image ideals of certain homogeneous locally nilpotent derivations of rank $2$.}

\smallskip

 \noindent
 {\small {{\bf Keywords}. Polynomial Rings, Homogeneous Locally Nilpotent Derivations, Triangularizable derivations, Non-triangularizable derivations, Image ideals.}
 
 }

\section{Introduction}

Throughout this article $k$ denotes a field of characteristic zero.
 Let $R$ be an integral domain containing $k$ and $S$ a subring of $R$ containing $k$. 
 Then $\LND_S(R)$ denotes the set of all locally nilpotent derivations (LNDs) $D$ on $R$ such that $D|_{S}=0$. 
 By $R^{[n]}$ we denote a polynomial ring in $n (\geqslant 1)$ variables over $R$.
  For an LND $D \in \LND_R(R^{[n]})$, the rank of $D$, denoted by $rank(D)$, is defined as the smallest positive integer $r$ such that there exists a coordinate system $\{V_{1}, \cdots, V_{n}\}$ of $R^{[n]}$  for which $DV_i \neq 0$ for $1\leqslant i \leqslant r$, and $DV_i=0$ for $i>r$. By homogeneous LND $D$ on $B=k[X,Y,Z](=k^{[3]})$, we mean $D$ is homogeneous with respect to the standard grading $(1,1,1)$ on $B$ (see \thref{def}(vi)), and for any $f\in B$, $\deg(f)$ denotes its degree with respect to the above grading. 

\medskip
By a theorem of Rentschler (see \cite{ren}) it follows that there is no non-zero LND having full rank on $k^{[2]}$. 
For $n\geqslant 3$, Freudenburg has constructed examples of LNDs (homogeneous) of full rank on $k^{[n]}$ (see \cite[Sections 2 and Section 3]{GF}). 
However several mysteries about rank $2$ LNDs on $k^{[n]}$ $(n \geqslant 3)$, are still unsolved. 
In \cite{df}, Daigle and Freudenburg have done an extensive study on rank $2$ LNDs on $k^{[n]}$ for $n \geqslant 3$ and they have given the first example of an irreducible non-triangularizable LND of rank $2$ on $k^{[3]}$ (\cite[Example 4.3]{df}). 
Further, in \cite{dtr96} and \cite{dtr09}, Daigle has given two different characterizations of triangularizable LNDs on $k^{[3]}$. However, when the LNDs are homogeneous, from these two characterizations it is not clearly understood whether the degree of the LNDs have any connection to their triangularizability property.

\medskip 

In this article, our objective is to understand the structure of homogeneous non-triangularizable LNDs of rank $2$ on $k[X,Y,Z]$ and their kernels.
  We first observe that if an LND $D$ of rank $2$ on $k[X,Y,Z]$ is irreducible and homogeneous, then its degree (defined in section 2) plays a crucial role in determining whether $D$ is triangularizable.
   More precisely, we show that for a prime number $p$, every irreducible homogeneous LND of rank $2$ and degree $p-2$ is triangularizable (\thref{ctr}). Since $4$ is the smallest non-prime integer, the smallest possible degree of an irreducible homogeneous non-triangularizable LND can be 2 (see \thref{r1}). Note that $2$ is an integer of type $pq-2$, where $p,q$ are prime numbers.   
    
    \smallskip
    Our study on non-triangularizable LNDs is motivated by the above observations.  
    Over an algebraically closed field $k$, and for prime numbers $p$ and $q$, not necessarily distinct, we characterize irreduclible homogeneous non-triangularizable LNDs of rank $2$ and degree $pq-2$ on $k[X,Y,Z]$ in \thref{ntr}.
     In particular, we characterize irreducible homogeneous non-triangularizable LNDs of rank $2$ and of the smallest possible degree on $k[X,Y,Z]$.
     
     \medskip
     In a recent work, Khaddah, Kahoui and Ouali have shown that for a PID $R$, $R^{[2]}$ is a free module with a $D$-basis (see \thref{dbasis}) over $ker(D)$ for any locally nilpotent $R$-derivation $D$ on $R^{[2]}$ (see \cite{kah}).
     That means the image ideals (see \thref{def}(iv)) of any locally nilpotent $R$-derivation $D$ on $R^{[2]}$ are principal ideals and 
     in particular, the image ideals of every rank $2$ LND on $k[X,Y,Z]$ are principal.
     However, no study has been done to describe the generators of the image ideals.
     In section 4 of this article, we show that our results in section 3 can be applied to study the generators of the image ideals of certain LNDs of rank $2$. 
     To be specific, for a field $k$ and for an homogeneous LND $D$ on $k[X,Y,Z]$ of rank $2$, we have described the generator of the $n$-th image ideal $I_n:=D^n(k[X,Y,Z]) \cap ker(D)$ for every integer $n \geqslant 0$ where  $D$ is either triangularizable or irreducible non-triangularizable of degree $pq-2$, $p$ and $q$ being primes (see \thref{trf} and \thref{ntrf}). 
%
%
%
%
 
 
 \medskip
 
 In the next section we record some well known results, definitions and properties of LNDs.



\section{Preliminaries}

 We first recall some definitions and basic properties of locally nilpotent derivations (cf. \cite{GFB}). 
 
 \begin{definition}\thlabel{def}
 	Let $B$ be an integral domain containing $k$, $D$ a non-trivial locally nilpotent derivation on $B$, and $A=ker(D)$.
 	\begin{enumerate}
 		\item[\rm{(i)}] \rm{ An element $r \in B$ is called a {\it local slice} of $D$, if $Dr \in ker D \setminus \{0\}$.}
 		
 		\item[\rm(ii)] $D$ defines a degree function $\mu:=\deg_{D}$ on $B$ such that  $\deg_{D}(0)=-\infty$ and for every nonzero $b \in B$ 
 		$$
 		\mu(b) (= \deg_{D}(b))=max\{n \in \mathbb{N}\,\,\,| D^n(b) \neq 0\}.
 		$$
 		
 		\item[\rm{(iii)}] \rm{Let $\mu$ be the degree function on $B$ induced by $D$. For a non-negative integer $n$, we define the {\it $n$-th degree $A$-module} with respect to $\mu$, as follows : 
 			$$	
 			\mathscr{F}_{n}=\{b \in B \mid \mu(b) \leqslant n\}.
 			$$ }  
 		
 		\item[\rm{(iv)}]  \rm{For a non-negative integer $n$, the {\it $n$-th image ideal} of $D$ is defined as $I_n:=D^nB \cap A$. }
 		
 		\item[\rm{(v)}] 	\rm{$D$ is said to be {\it irreducible} if the ideal $\left(DB\right)$ is not contained in any proper principal ideal of $B$. }
 		
%
 		\item[\rm{(vi)}] \rm{Let $G$ be a totally ordered Abelian group and $B$ a $G$-graded ring such that $B=\bigoplus_{i \in G}B_i$ be the $G$-graded structure on $B$. Then $D$ is said to a {\it homogeneous derivation} on $B$, if there exists some $d \in G$, such that $DB_i \subseteq B_{i+d}$ for every $i \in G$, and $\deg_G(D):=d$ is said to be the {\it degree} of $D$. If $G=\mathbb{Z}$, then $\deg_{\mathbb{Z}}(D)$ will be denoted by $\deg(D)$.} 
 		
 		\item[\rm{(vii)}]	\rm{For $B=k^{[n]}$, $D$ is said to be {\it triangularizable} if there exists a system of coordinates $\{X_1,\ldots,X_n\}$ of $B$ such that $DX_1 \in k$ and $DX_i \in k[X_1,\ldots,X_{i-1}]$, for every $i \geqslant 2$.} 
 		
 		\item[\rm{(viii)}] \rm{Let $B=k[X_1,\ldots,X_n]$ and $f_1,\ldots,f_{n-1} \in B$. For $\underline{f}:=(f_1,\ldots,f_{n-1})$, the {\it Jacobian derivation} $\Delta_{\underline{f}}$ on $B$ is defined as follows: 
 			$$
 			\Delta_{\underline{f}}(g):=\frac{\partial(f_1,\ldots,f_{n-1},g)}{\partial(X_1,\ldots,X_n)},
 			$$
 			for every $g \in B$.}
 		
 		\item[\rm{(ix)}] \rm{ A collection $\{B_{n}\,\,| \, n \in \mathbb{Z}\}$ of $k$-subspaces of $B$ is said to be a {\it proper $\mathbb{Z}$-filtration} if
 			
 			\begin{itemize}
 				\item [\rm(a)] $B_{n} \subseteq B_{n+1}$ for every $n \in \mathbb{Z}$.
 				
 				\item[\rm(b)]  $B= \bigcup_{n \in \mathbb{Z}} B_{n}$.
 				
 				\item[\rm(c)] $\bigcap_{n \in \mathbb{Z}} B_{n}=\{0\}$.
 				
 				\item[\rm(d)]  $(B_{n} \setminus B_{n-1}).(B_{m}\setminus B_{m-1}) \subseteq B_{m+n} \setminus B_{m+n-1}$ for all $m,n \in \mathbb{Z}$.
 		\end{itemize}} 
 	\end{enumerate}
 	
 \end{definition}
 
 	\begin{lem}\thlabel{prop}
 	Let $B$ be an integral domain containing $k$, $D$ a non-trivial locally nilpotent derivation on $B$, and $A=ker(D)$. Then the following statements hold:
 	
 	\begin{itemize}
 		
 		\item[\rm(i)] $A$ is a factorially closed subring of $B$ and hence algebraically closed. 
 		
 		\item[\rm(ii)] For an element $r \in B \setminus A$ such that $D^2r=0$, we have $B_{Dr}=A_{Dr}[r]=A_{Dr}^{[1]}$. 
 		
 		
 		
 		\item[\rm{(iii)}]  Let $S$ be a multiplicatively closed subset of $A\setminus \{0\}$. Then $D$ will induce a locally nilpotent derivation $S^{-1}D$ on $S^{-1}B$ and $ker(S^{-1}D)=S^{-1}A$. 
 		
 		\item[\rm{(iv)}] Let $\overline{k}$ be an algebraic closure of $k$ and $\overline{D}$ denotes its natural extension to $\overline{B}:=B \otimes_{k} \overline{k}$. Then $\overline{D} \in \LND(\overline{B})$ and $ker(\overline{D})= A \otimes_{k} \overline{k}$.
 	\end{itemize}

 \end{lem}

 Next we recall some known results. The first one is the  Rentschler Theorem (\cite{ren}).
\begin{thm}\thlabel{ren}
	Let $D$ be a non-zero locally nilpotent derivation on $k[X,Y]$. Then there exist $p(X) \in k[X]$ and a tame automorphism $\sigma$ of $k[X,Y]$ such that $\sigma D \sigma^{-1}= p(X)\frac{\partial}{\partial Y}$.
\end{thm}


Next we quote an important result of Miyanishi (\cite{miya}).

\begin{thm}\thlabel{mthm}
	Let $D$ be a non-zero locally nilpotent derivation on $k[X,Y,Z]$. Then $ker(D)=k^{[2]}$.
\end{thm}

The following theorem is due to Zurkowski (\cite{zur}).

\begin{thm}\thlabel{zu}
	Let $D$ be a nonzero homogeneous locally nilpotent derivation with respect to some positive grading $\omega$ on $k[X,Y,Z]$ and $A=ker(D)$. Then there exist homogeneous polynomials $F, G$ with respect to that grading such that $A=k[F,G]$.
\end{thm}

Next we record two easy lemmas.

\begin{lem}\thlabel{htr}
	Let $D$ be a homogeneous triangularizable locally nilpotent derivation of degree $d$ on $k[U,V,W]$ with respect to the standard weights $(1,1,1)$. Then there exists a linear system of variables $\{X,Y,Z\}$ with respect to which $D$ is triangular.  
\end{lem}

\begin{proof}
	Since $D$ is triangularizable, there exists a system of variables $\{U_1,U_2,U_3\}$ such that 
	\begin{equation}\label{u}
		DU_1=0,\, DU_2=f(U_1),\, DU_3=g(U_1,U_2),
	\end{equation}
	for some $f \in k^{[1]}$ and $g \in k^{[2]}$. Suppose $L_i$ denotes the linear part of $U_i$ for every $i, 1 \leqslant i \leqslant 3$. Since $D$ is homogeneous of degree $d$, from \eqref{u}, it follows that 
	$$
		DL_1=0,\, DL_2=\lambda L_1^{d+1},\, DL_3=\widetilde{g}(L_1,L_2),
	$$
	for some $\lambda \in k^*$ and homogeneous polynomial $\widetilde{g}$ of degree $d+1$. Therefore, the assertion holds for $\{X,Y,Z\}=\{L_1,L_2,L_3\}$.
	\end{proof}

\begin{lem}\thlabel{irr}
		Let $B$ be an affine $k$-algebra which is a UFD and $D$ a non-trivial locally nilpotent derivation on $B$. 
	Let $\overline{k}$ be an algebraic closure of 
	$k$ and $\overline{D}$ 
	denotes the natural extension of $D$ on $\overline{B}$, where $\overline{B}=B \otimes_{k} \overline{k}$.
	If $D$ is irreducible then so is $\overline{D}$.
\end{lem}

\begin{proof}
	
	Let $J=(DB)$ and $\overline{J}=(\overline{D} (\overline{B}))$. Clearly, $\overline{J}=J \overline{B}$. Suppose, if possible, $D$ is irreducible but $\overline{D}$ is not. 
	Then there exists $\overline{b} \in \overline{B}$ and a prime 
	ideal $\p$ of $\overline{B}$ such that $\hgt \p =1$ and 
	$\overline{J} \subseteq (\overline{b}) \subseteq \p$; and hence $J \subseteq \p \cap B= \q$. Since $B \subseteq \overline{B}$ is a flat extension, it satisfies the going down property (cf. \cite[5.D, Theorem 4]{mat}), and hence $\hgt \q=1$. Now, since $B$ is a UFD, $\q$ is principal, which contradicts that $D$ is irreducible. Hence the result follows. 
\end{proof}


The following result of Daigle (\cite[Corollary 2.5]{daig}) describes the structure of  LNDs on $k^{[3]}$ in terms of the Jacobian derivation.

\begin{thm}\thlabel{dai}
	Let $B=k^{[n]}$ and $D \in \LND(B)$. Suppose that $\{f_{1},\cdots,f_{n-1}\}$ is a set of algebraically independent elements in $B$ such that $ker(D)=k[f_{1},\ldots,f_{n-1}]=k^{[n-1]}$. Then $\Delta_{(f_{1},\ldots,f_{n-1})} \in \LND(B)$ and $D=a \Delta_{(f_{1},\ldots,f_{n-1})}$, for some $a \in ker(D)$.
\end{thm}

We now recall the concept of Newton polygon. Let $A$ be a commutative $k$-domain and $B=A[X,Y]$ be a $\mathbb{Z}^2$-graded domain with respect to the following weights:
\begin{equation}\label{*}
	wt(X)=(1,0),\, wt(Y)=(0,1)
\end{equation} 
and $wt(a)=(0,0)$ for all $a \in A$. We record the definition of the Newton polygon of $f \in B$ below.
\begin{definition}
\em{	Let $f \in B:=A[X,Y]$. The {\it Newton polygon} of $f$ is denoted by $Newt_{\mathbb{Z}^2}(f)$ and is defined to be the convex hull in $\mathbb{R}^2$ of the following set:
	$$
	S=\left\{ (i,j) \in \mathbb{Z}^2 \,\,\,|\,\,\, f=\sum a_{ij}X^iY^j, a_{ij}\in A\setminus \{0  \}\right\} \cup \left\{(0,0)\right\} 
	$$}
\end{definition}

The following well known result (\cite[Theorem 4.5]{GFB}) is needed in section 3 of this note.

\begin{thm}\thlabel{np}
	Let $A$ be a rigid affine $k$-domain, i.e., there is no non-trivial LND on $A$. Suppose $B=A[X,Y]$ is a $\mathbb{Z}^2$-graded domain with respect to the weights defined in \eqref{*} and $D$ a non-trivial locally nilpotent derivation on $B$. Then, for $f\in ker(D)\setminus A$, $Newt_{\mathbb{Z}^2}(f)$ is a triangle with vertices $(0,0), (m,0), (0,n)$ where $m,n \in \mathbb{N}$ and $m \mid n$ or $n \mid m$.
\end{thm}

\begin{rem}\thlabel{grnp}
	\em{ By \thref{np}, it follows that if $A=k$, $B=k[X,Y]$, and $f\in ker(D)\setminus k$ is such that $n=\deg_X(f) > \deg_Y(f)=m$, then $f$ has the following form:
	$$
	f= \tilde{f}(X)+ \sum_{j=1}^{m-1} f_j(X) Y^j+\alpha Y^m, 
	$$
	where $\deg_X(\tilde{f})=n$, $\alpha \in k^*$, $n=mq$ for some positive integer $q$ and $\deg_X (f_j) \leqslant n-jq$.}
\end{rem}

Let $B$ be an affine domain with a proper $\mathbb{Z}$-filtration $\{B_{n} \mid n \in \mathbb{Z}\}$. If $D$ is a non-zero LND on $B$ such that for all $n \in \mathbb{Z}$, $D(B_n) \subset B_{n+t}$ for some $t \in \mathbb{Z}$ (i.e., $D$ respects the filtration on $B$), then it will induce $gr(D) \in \LND(gr(B))$, where $gr(B)$ is the associated graded ring of $B$ with respect to the given filtration (\cite[pg. 10]{GFB}). Now suppose that $\rho : B \rightarrow gr(B)$, denote the natural map defined by $\rho(b)=b+B_{i-1}$, where $b \in B_{i}\setminus B_{i-1}$ for some $i \in \mathbb{Z}$.
Then we have the following result due to Derksen {\it et al.} (\cite{dom}). For reference one can see \cite[Theorem 2.6]{cra}.

\begin{prop}\thlabel{gd}
	Let $B, G, \rho$ and $D$ be the same as mentioned in the above paragraph. Then $gr(D) \neq 0$ and $\rho(ker(D)) \subset ker(gr(D))$.   
\end{prop}

The following result is a special case of a result of Daigle \cite[Theorem 1.7]{dtame}. This result first appeared in the thesis of Wang (see \cite{twang}).

\begin{thm}\thlabel{dtame}
	Let $B$ be a $\mathbb{Z}$-graded affine $k$-domain. Then every non-zero $D \in \LND(B)$ respects the $\mathbb{Z}$-filtration induced by the grading.
	
\end{thm}

We also fix a notation which will be used in the note. Let $f \in k[U,V,W]$. Then $f_U,f_V,f_W$ will denote the partial derivatives of $f$ with respect to $U,V,W$ respectively.

\section{Homogeneous locally nilpotent derivations of rank two}

 {\bf $(*)$} We first fix a few notation for this section.
Throughout this section, unless specified, $D$ denotes an irreducible homogeneous LND of rank $2$ on $k[U,V,W] $ such that $\deg(D)=d\, (\geqslant 0)$ with respect to the standard weights $(1,1,1)$. 
Since $rank(D)=2$, without loss of generality we can assume that $DU=0$ and $ker(D)=k[U,P]$ for some homogeneous polynomial $P \in k[U,V,W]$ (cf. \thref{mthm} and \thref{zu}). As $D$ is irreducible, multiplying $P$ by a suitable constant in $k$ we have $D=\Delta_{(U,P)}$ (cf. \thref{dai}). Hence 
$ DU=0 , DV=-P_{W} ,  DW=P_{V} $.
If $\deg(D)=d$, then
$P$ is a homogeneous polynomial of degree $d+2$.

\smallskip
 
First we observe some results (Lemmas \ref{sa}, \ref{sb} and \ref{tr}) which exhibit the structure of $P$.
\begin{lem}\thlabel{sa}
	Let $D$ and $P$ be the same as in the paragraph $(*)$. Then there exists a linear system of variables $\{X,Y,Z\}$ of $k[U,V,W]$ such that upto multiplication by a unit, $P$ has the form 
	$$
	 Y^{d+2}+ Xq(X,Y,Z)
	$$ 
	 where $q(X,Y,Z)$ is a degree $d+1$ homogeneous polynomial, $DX=0$ and $0< \deg_{D}(Y)< ~\deg_{D}(Z).$
\end{lem}

\begin{proof}
	Since $DU=0$, $D$ induces an LND $\overline{D}= D \,(\text{mod~} U)$ on $\frac{k[U,V,W]}{(U)} \cong k[V,W]$.  Note that $\overline{D}$ is a non-zero homogeneous LND of degree $d$. By \thref{ren}, there exists a system of variables $\{V_{1}, V_{2}\}$ of $k[V,W]$ such that $\overline{D}=f(V_1) \frac{\partial}{\partial V_2}$ for some $f \in k^{[1]}$, and $ker(\overline{D})= k[V_{1}]$. Further, $\{V_1,V_2\}$ can be chosen to be linear in $V$ and $W$, as $\overline{D}$ is homogeneous.
	Since $P$ is a homogeneous polynomial of degree $d+2$, if $\overline{P}=P\, (\text{mod~} U)$, then $\overline{P}$ is a homogeneous polynomial in $ker(\overline{D})=k[V_{1}]$ of degree $d+2$. Therefore, with respect to the linear system of variables $\{U,V_{1},V_{2}\}$ of $k[U,V,W]$, we get a homogeneous polynomial $q$ of degree $d+1$, such that upto multiplication by a unit, $P=V_{1}^{d+2}+U q(U,V_{1},V_{2})$. 
	
	Suppose $\deg_{D}(V_{1}) \geqslant \deg_{D}(V_{2})$. From the structure of $P$, it is clear that $\deg_{D}(P)=(d+1)\deg_{D}(V_{1})$, and hence $\deg_{D}(V_{1})=0$. This contradicts the fact that $rank(D)=2$. Therefore, we must have $\deg_{D}(V_{1})< \deg_{D}(V_{2})$. Hence renaming the coordinate system $\{U,V_{1},V_{2}\}$ as $\{X,Y,Z\}$, the result follows.
\end{proof}

\begin{lem}\thlabel{sb}
	Let $D$ and $P$ be the same as in the paragraph $(*)$. Then there exists a linear system of variables $\{X,Y,Z\}$ of $k[U,V,W]$ such that $0=\deg_{D}(X) < \deg_{D}(Y) < \deg_{D}(Z)$, and upto multiplication by a unit, the polynomial $P$ has the following form:
	
	\medskip
	\noindent
	(i) For $d=0$, $P=Y^2+XZ$.
	
	\medskip
	\noindent
	(ii) For $d \geqslant 1$,
	$$
	P=Y^{d+2}+ Xf_{d+1}(X,Y) + Xf_{d}(X,Y)Z+ \dots+Xf_{i+2}(X,Y)Z^{d-i-1} +\beta X^{i+2}Z^{d-i}
	$$ 
    where $0 \leqslant i \leqslant d-1$ such that $d-i \mid d+2$, $\beta \in k^{*}$ and $f_{j}(X,Y)$ is a  homogeneous polynomials of degree $j$, for every $j$, $i+2 \leqslant j \leqslant d+1$.	
\end{lem}

\begin{proof}
	In \thref{sa}, we see that $P=\alpha^{\prime } P^{\prime}$ such that
	$$
	P^{\prime}=Y^{d+2}+Xq(X,Y,Z),
	$$ 
    and $\alpha^{\prime} \in k^*$, where $0=\deg_{D}(X)< \deg_{D}(Y) < \deg_{D}(Z)$ and $q(X,Y,Z)$ is a homogeneous polynomial of degree $d+1$. We rename $P^{\prime}$ as $P$ and proceed. 
    
 	Now for $d=0$, $P=Y^2+X(\alpha X+\beta Y+\gamma Z)$, for some $\alpha, \beta, \gamma \in k$. If $\gamma =0$, then $P=Y^2+\alpha X^2+ \beta XY$. Since $X \in ker(D)$, it follows that $Y(Y+\beta X) \in ker(D)$, and hence $Y \in ker(D)$, as $ker(D)$ is factorially closed (\thref{prop}(i)). But this contradicts that $rank(D)=2$.
 	 Therefore, $\gamma \in k^{*}$, and with respect to the system of variables $\{X,Y, \alpha X+ \beta Y+\gamma Z\}$, we have $P=Y^2+XZ$.
	
	We now consider the case $d \geqslant 1$. By \thref{prop}(iii), $D$ extends to $D^{\prime} \in \LND \left(k(X)[Y,Z]\right)$ such that $ker(D^{\prime})=k(X)[P]$.
	Therefore, by \thref{np}, either $\deg_{Y}P\mid \deg_{Z}P$ or $\deg_{Z}P \mid \deg_{Y}P$ and $P$ is almost monic in $Z$ as a polynomial in $k(X)[Y,Z]$. 
	As $\deg_{Y}P > \deg_{Z}P$, we have $\deg_{Z}P \mid \deg_{Y}P$. 
	Again $\gcd(d+1,d+2)=1$ for $d \geqslant 1$. Therefore, expanding the expression of $P$ we get
	$$
	P= Y^{d+2}+ Xf_{d+1}(X,Y)+ Xf_{d}(X,Y)Z+ \dots + Xf_{i+2}(X,Y)Z^{d-i-1}+\beta_1 X^{i+2} Z^{d-i},
	$$
	such that $\beta_1 \in k^{*}$, $d-i \mid d+2$ for some $i$, where $0 \leqslant i \leqslant d-1$ and every polynomial $f_{j}(X,Y)$ is a homogeneous polynomials of degree $j$, where $i+2 \leqslant j \leqslant d+1$.
\end{proof}

The next lemma gives the structure of $P$ for irreducible, homogeneous triangularizable LNDs on $k^{[3]}$. Here we mention that in \cite[Corollary 5.2]{dtr09}, Daigle has shown that for $D$ to be a triangularizable LND on $k^{[3]}$, it is necessary and sufficient that a coordinate of $k^{[3]}$ is a local slice for $D$. The following \thref{tr} can be deduced from this result of Daigle. However, we give an independent proof using the definition of triangularizable derivations.  

\begin{lem}\thlabel{tr}
Let $D$ and $P$ be the same as in the paragraph $(*)$. Then $D$ is triangularizable if and only if there exists a system of variables $\{X,Y,Z\}$ which is linear in $\{U,V,W\}$ such that
	 $0=\deg_{D}(X) < \deg_{D}(Y) < \deg_{D}(Z)$ and $D=\gamma \Delta_{(X,P)}$ for some $\gamma \in k^{*}$, where 
	$$
	P= Y^{d+2}+ Xf_{d+1}(X,Y)+\beta X^{d+1}Z,
	$$ 
	 $f_{d+1}(X,Y)$ is a homogeneous polynomial of degree $d+1$ and $\beta \in k^{*}$.
	Moreover, $\deg_{D}(Y)=1$ and $\deg_{D}(Z)=d+2$.
\end{lem}
\begin{proof}
	It is easy to see that for the given structure of $P$ and $\gamma \in k^*$, $D=\gamma \Delta_{(X,P)}$ is triangularizable.
	
	Conversely, suppose $D$ is triangularizable.  
	%
	Since $D$ is homogeneous, by \thref{htr}, there exists a linear system of variables $\{X,Y,Z\}$ such that
	$$
	DX=0,\,\,DY=\mu  X^{d+1},\,\, DZ=g(X,Y),
	$$
	 for some $\mu \in k^{*}$ and a homogeneous polynomial $g$ of degree $d+1$. Therefore, $\deg_{D}(Y)=1$. 
	Now by \thref{mthm} and \thref{zu}, $ker(D)=k[X,P_{1}]$ for some homogeneous polynomial $P_1$. Since $D$ is irreducible, by \thref{dai},
	with respect to the coordinate system $\{X,Y,Z\}$, $D=\Delta_{(X,P_{1})}$, upto multiplication by a nonzero constant. Therefore, 
	$$
	DY=-\lambda(P_1)_Z={\mu} X^{d+1}
	$$ for some $\lambda \in k^{*}$, and hence
	$P_{1}= \lambda^{-1}(\widetilde{f}(X,Y)-\mu X^{d+1}Z),$  
	where $\widetilde{f}(X,Y)$ is a homogeneous polynomial of degree $d+2$.  Now using the fact that $P_{1}$ is irreducible, we get that $P_{1}$ is of the form  
	$$
	P_{1}=\alpha (Y^{d+2}+ Xf_{d+1}(X,Y) + \beta X^{d+1}Z)
	$$ 
	for some $\alpha, \beta \in k^{*}$ and a homogeneous polynomial $f_{d+1}(X,Y)$ of degree $d+1$. Therefore, we obtain that $D=\gamma \Delta_{(X,P)} $ for some $\gamma \in k^{*}$, and $ker(D)=k[X,P]$ where $P$ is in the desired form.
	
	Now
	$$
	DZ=\gamma P_{Y}= \gamma \left((d+2)Y^{d+1}+ X(f_{d+1})_Y \right).
	$$
	Since $X \in ker(D)$ and $\deg_{D}(Y)=1$, we have 
	$
	\deg_{D}(Z)= (d+1)\deg_{D}(Y)+1=d+2
	$.  
\end{proof}

\begin{rem}\thlabel{rk1}
	\em
	{Let $D$ be an irreducible homogeneous LND of rank $2$ on $B=k[U,V,W]$ such that $ker(D)=k[U,P]$. By \thref{sa}, we see there exists a linear system of variables $\{ X,Y,Z\}$ of $B$ such that
	\begin{equation}\label{a}
		0=\deg_{D}(X)< \deg_{D}(Y) < \deg_{D}(Z).
    \end{equation}
  Let $\overline{k}$ be an algebraic closure of $k$ and $D$ extends to $\overline{D} \in \LND(\overline{k}[X,Y,Z])$. If $\overline{D}$ is triangularizable, then by \thref{tr}, there exists a system of variables $\{X_{1},Y_{1},Z_{1}\}$ of $\overline{k}[X,Y,Z]$, which are linear in $X,Y,Z$ such that 
  \begin{equation}\label{b}
  	0=\deg_{\overline{D}}(X_{1})< \deg_{\overline{D}}(Y_{1})< \deg_{\overline{D}}(Z_{1}),
  \end{equation}
$ker(\overline{D})=\overline{k}[X_1,P]$, where $P= Y_{1}^{d+2}+ X_{1}f_{d+1}(X_{1},Y_{1})+\beta X_{1}^{d+1}Z_{1} \in \overline{k}[X_1,Y_1,Z_1]$, for some homogeneous polynomial $f_{d+1}(X_{1},Y_{1})$ and $\beta \in \overline{k}^{*}$. From \eqref{a} and \eqref{b}, it is easy to see that $X_{1}=a_{11}X,\, Y_{1}=a_{21}X+a_{22}Y$ and $Z_{1}=a_{31}X+a_{32}Y+a_{33}Z$, where every $a_{ij} \in \overline{k}$ and $a_{11},a_{22},a_{33} \neq 0$. Since $P \in k[X,Y,Z]$, upto multiplication by a non-zero constant in $k$, $P=Y^{d+2}+ Xg_{d+1}(X,Y)+\gamma X^{d+1}Z$, where $\gamma \in k^{*}$ and $g_{d+1}(X,Y)$ is a homogeneous polynomial in $k[X,Y]$ of degree $d+1$. Hence by \thref{tr}, $D$ must be triangularizable.
   }
\end{rem}

As an application of the above two Lemmas \ref{sb} and \ref{tr}, we get the following result.

\begin{cor}\thlabel{ctr}
	Let $p$ be a natural number and $D$ an irreducible homogeneous LND of rank $2$ and degree $p-2$ on $k[U,V,W]$. If $p$ is a prime, then $D$ is triangularizable. 
	
\end{cor}	
\begin{proof}
 Since $p$ is a prime, by \thref{sb} there exists a linear system of variables $\{X,Y,Z\}$ of $k[U,V,W]$ such that $D=\gamma \Delta_{(X,P)}$ and 
	$P=Y^{p}+Xf_{p-1}(X,Y)+\beta X^{p-1}Z,$ 
	where  $\gamma, \beta \in k^{*}$ and $f_{p-1}(X,Y)$ is homogeneous polynomial of degree $p-1$. Hence by \thref{tr}, $D$ is triangularizable.
\end{proof}

\begin{rem}\thlabel{r1}
	\em{In view of \thref{ctr}, it can be noticed that the smallest possible degree of a non-triangularizable irreducible homogeneous LND of rank $2$ on $k[U,V,W]$
	is $2$ which comes from the case ``$p$ is not a prime". The next theorem gives a structure of $P$ for such LNDs. Specifically,
	it establishes the structure of $P$ for irreducible homogeneous LNDs of rank $2$ and degree $pq-2$ on $k[U,V,W]$, where $p,q$ are prime numbers, not necessarily distinct.}
	
\end{rem}

\begin{thm}\thlabel{ntr}
	Let $k$ be an algebraically closed field and $p,q$ are prime numbers, not necessarily distinct. 
	Suppose $D$ and $P$ are as in the paragraph $(*)$ such that $\deg(D)=d=pq-2$. Then $D$ is not triangularizable if and only if there exists a system of variables $\{X,Y,Z\}$ linear in $\{U,V,W\}$ and a homogeneous polynomial $h(X,Y)$, monic in $Y$, such that $D= \gamma \Delta_{(X,P)}$ where $\gamma \in k^{*}$ and $P$ takes the following form where the roles of $p$ and $q$ are interchangeable:
	$$
	P=T^{p}+ c_{1}X^{q}T^{p-1}+\dots+c_iX^{iq}T^{p-i}+\dots+c_{p-1}X^{pq-q}T+  
	c_{p}X^{pq-1}Y,
	$$ 
	where $T=h(X,Y)+X^{q-1}Z$, $\deg(h(X,Y))=q$, $c_{i} \in k$\text{~for~}$i,1 \leqslant i \leqslant p$ and $c_{p} \neq 0$.
	Moreover, $\deg_{D}(Y)=p$, $\deg_{D}(Z)=pq$ and $T$ is a local slice for $D$.
\end{thm}
\begin{proof}
	Suppose $D$ is not triangularizable. By \thref{sb} and \thref{tr}, there exists a system of variables $\{X,Y,Z\}$ which are linear in $\{U,V,W\}$ such that 
	$D=\gamma \Delta_{(X,P)}$, 
	where $\gamma \in k^{*}$ and $P$ has either of the following forms:
	\begin{equation}\label{14}
		P= Y^{pq}+Xf_{pq-1}(X,Y)+XZf_{pq-2}(X,Y)+\dots+ XZ^{i-1}f_{pq-i}(X,Y)+\dots+\beta X^{pq-p}Z^{p}
	\end{equation}
	where $0=\deg_{D}(X) < \deg_{D}(Y) < \deg_{D}(Z)$, $\beta \in k^{*}$ and $f_{pq-i}(X,Y)$ is a homogeneous polynomial of degree $pq-i$ for $1 \leqslant i \leqslant p$ or
	
		\begin{equation}\label{q}
		P= Y^{pq}+X\widetilde{f}_{pq-1}(X,Y)+XZ\widetilde{f}_{pq-2}(X,Y)+\dots+XZ^{i-1}\widetilde{f}_{pq-i}(X,Y)+\dots+\widetilde{\beta} X^{pq-q}Z^{q}
	\end{equation}
	where $0=\deg_{D}(X) < \deg_{D}(Y) < \deg_{D}(Z)$, $\widetilde{\beta} \in k^{*}$ and $\widetilde{f}_{pq-i}(X,Y)$ is a homogeneous polynomial of degree $pq-i$ for $1 \leqslant i \leqslant q$. 
	Therefore without loss of generality we take $P$ as in \eqref{14} and proceed. 
	
	Now $D$ extends to an LND $\widetilde{D}$ of $k(X)[Y,Z]$  and $P \in ker(\widetilde{D})$ (cf. \thref{prop}(iii)).
By \thref{grnp}, $\deg_{Y}(f_{pq-i}(X,Y)) \leqslant pq-(i-1)q$.
 We consider the following grading on $k(X)[Y,Z]$:
$$
gr_{1}(Y)=1,\,\,\,\, gr_{1}(Z)=q.
$$
 Therefore, if $P_1$ denotes the highest degree homogeneous summand of $P$ with respect to $gr_1$, then
	$$
	P_{1}=Y^{pq}+\gamma_{1}Y^{pq-q}(X^{q-1}Z)+\dots+\gamma_{p-1}Y^{q}(X^{q-1}Z)^{p-1}+\beta(X^{q-1}Z)^p,
	$$
	where $\gamma_{i} \in k$, for $ 1 \leqslant i \leqslant p-1$.
	As $k$ is algebraically closed, we have 
	$$
	P_{1}=\prod_{i=1}^{p}(Y^{q}+\alpha_{i}X^{q-1}Z),
	$$
	where $\alpha_{i} \in k$, for $ 1 \leqslant i \leqslant p$. By \thref{dtame}, $gr_1(\widetilde{D}) \in \LND(k(X)[Y,Z])$, and
	by \thref{gd}, $P_1 \in ker(gr_1(\widetilde{D}))$. Now as $ker(gr_1(\widetilde{D}))$ is factorially closed (cf. \thref{prop}(i)), if there exist $i,j$ such that $\alpha_{i} \neq \alpha_{j}$, then it follows that $Y,Z\in ker(gr_{1}(\widetilde{D}))$. But then $gr_1(\widetilde{D})=0$ which is a contradiction (cf. \thref{gd}). 
	Therefore, we have $P_{1}=(Y^{q}+\alpha X^{q-1}Z)^{p}$, where $\alpha_{i}=\alpha \in k^{*}$ for every $i, 1 \leqslant i \leqslant p$. 
	We now rename $\alpha Z$ as $Z$. 
	For $Z_{1}=(Y^{q}+X^{q-1}Z)$, as $k(X)[Y,Z]=k(X)[Y,Z_1]$, we have the following form of $P$. 
	\begin{equation}\label{15}
		P=Z_{1}^{p}+XZ_{1}^{p-1}g^1_{q-1}(X,Y)+\dots+XZ_{1}^{p-j}g^1_{jq-1}(X,Y)+\dots+
		Xg^1_{pq-1}(X,Y),
	\end{equation}
	where $g^1_{jq-1}(X,Y)$ is a homogeneous polynomial of degree $jq-1$, $ 1 \leqslant j \leqslant p$ and $g^1_{pq-1}(X,Y) \neq 0$.
		
		Since $k(X)[Y,Z]=k(X)[Y,Z_{1}]$, if $\deg_{Y}(g^1_{pq-1}(X,Y)) \geqslant p$, then $p \mid \deg_{Y}(g^1_{pq-1}(X,Y)) $ (cf. \thref{np}). Therefore, $\deg_{Y}(g^1_{pq-1}(X,Y))=rp$ for some $r$, where $1 \leqslant r < q$. 
		By \thref{grnp}, $\deg_{Y}(g^1_{jq-1}(X,Y)) \leqslant rp-(p-j)r=jr$, for $ 1 \leqslant j \leqslant p$. 
		We now consider the following grading on $k(X)[Y,Z_{1}]$:
		$$
		gr_{2}(Y)=1,~~ gr_{2}(Z_{1})=r. 
		$$
		If $P_2$ denotes the highest degree homogeneous summand of $P$ with respect to $gr_2$, then by the similar arguments used for $P_{1}$, we get that $P_2=Z_2^p$ where $Z_{2}=(Z_1+\lambda X^{q-r}Y^r)=(Y^q+\lambda X^{q-r}Y^r +X^{q-1}Z)$ for some $\lambda \in k$. Hence 
	$$	
	P=Z_{2}^{p}+XZ_{2}^{p-1}g^2_{q-1}(X,Y)+\dots+XZ_2^{p-j}g^2_{jq-1}(X,Y)+\dots+
	Xg^2_{pq-1}(X,Y),
	$$ 
	where $g^2_{jq-1}(X,Y)$ is a homogeneous polynomial of degree $jq-1$, for $ 1 \leqslant j \leqslant p$, $g^2_{pq-1}(X,Y) \neq 0$ and $\deg_{Y}(g^2_{pq-1}(X,Y))< rp$. 
	
	Since $k(X)[Y,Z]=k(X)[Y,Z_{2}]$, if $\deg_{Y}(g^2_{pq-1}(X,Y)) \geqslant p$, then we can repeat the above process until we get the following form of $P$: 
	\begin{equation}\label{16}
		P=T^{p}+XT^{p-1}\widetilde{g}_{q-1}(X,Y)+\dots+XT^{p-j}\widetilde{g}_{jq-1}(X,Y)+\dots+
		 X\widetilde{g}_{pq-1}(X,Y),
	\end{equation}
	where $T=(h(X,Y)+X^{q-1}Z)$ for some homogeneous polynomial $h(X,Y)$ of degree $q$ which is monic in $Y$, $\widetilde{g}_{jq-1}(X,Y)$ is homogeneous polynomial of degree $jq-1$, for $ 1 \leqslant j \leqslant p$, $\widetilde{g}_{pq-1}(X,Y) \neq 0$ and $\deg_{Y}(\widetilde{g}_{pq-1}(X,Y))<p.$
  
   Now, since $k(X)[Y,T]=k(X)[Y,Z]$, and $\deg_{Y}(\widetilde{g}_{pq-1}(X,Y))<p$, by \thref{np} $\deg_{Y}(\widetilde{g}_{pq-1}(X,Y)) \mid p$, and therefore, $\deg_{Y}(\widetilde{g}_{pq-1}(X,Y))=1$.
	Hence we see that $P$ has the following form:
	\begin{equation}\label{17}
		P=T^{p}+ c_{1}X^{q}T^{p-1}+\dots+ c_iX^{iq}T^{p-i}+\dots+
		c_{p-1}X^{pq-q}T+ c_{p} X^{pq-1}Y
	\end{equation}
	where $c_{i} \in k$ for $i, 1 \leqslant i \leqslant p$ and $c_{p} \neq 0$ (cf. \thref{grnp}).
	Thus we have the desired form of $P$.
	\medskip
	
	Next, we investigate the $\deg_{D}$-values of $Y$ and $Z$. Note that
	$DY=-\gamma P_{Z}$ and $DZ=\gamma P_{Y}$.
	Since $ker(D)=k[X,P]$, $DY$ or $DZ$ can be in $ker(D)$ only if $P_Z$ or $P_Y$ is a polynomial entirely in $X$. But
	 from the expression of $P$ it is clear that this is not possible. Hence $DY$ and $DZ$ are not in $ker(D)$ and therefore, $\deg_{D}(Y) > 1$ and $\deg_{D}(Z)>1$. Now it is easy to check that
	$$
	DT=D(h(X,Y)+X^{q-1}Z)=\gamma \left(-h_{Y}P_{Z}+X^{q-1}P_{Y}\right)=\gamma c_{p}X^{pq+q-2}.
	$$ 
	Hence, $T=(h(X,Y)+X^{q-1}Z)$ is a local slice of $D$. Since $\deg_{D}(h(X,Y)+X^{q-1}Z)=1$ and $h(X,Y)$ is monic in $Y$, we have $\deg_{D}(Z)=\deg_{D}(h(X,Y))=q.\deg_{D}(Y)$. Since $\deg_{D}(P)=0$, and both $\deg_{D}(T)$ and $\deg_{D}(Y)$ are non-zero, from \eqref{17} we have $\deg_{D}(Y)=p.\deg_{D}(T)$. Hence, $\deg_{D}(Y)=p$ and $\deg_{D}(Z)=pq.$
\end{proof}

\section{An application: Finding generators of image ideals}

In this section we discuss an application of our main results. More precisely, we shall
use \thref{tr} and \thref{ntr} to find generators of the image ideals $I_n$'s of homogeneous triangularizable LNDs and irreducible homogeneous non-triangularizable LNDs of rank two and degree $pq-2$ on $k^{[3]}$, where $p,q$ are prime numbers.

\subsection{Definitions and preliminary results}

We start with some properties of the degree module $\mathscr{F}_{n}$ corresponding to an LND $D$ on $B=k[X,Y,Z]=k^{[3]}$. 
These properties have been described in \cite{Gfact}.

We first fix a few notation for the rest of this subsection. Let $r\in B$ be a local slice for $D$ and $Dr=f$, where $f \in A:=ker(D)$. Consider the following $A$ submodule $M$ of 
$\mathscr{F}_{n}$:
$$
M= \sum_{i=0}^{n}A.r^n.
$$
Let $M_{0}$ be an $A$-module such that $M \subseteq M_{0} \subseteq \mathscr{F}_{n}$ and 
$
M_{i}=\{b \in B \,\,|\,\, fb \in M_{i-1}\}
$ for every
$i \geqslant 1$. The following theorem (\cite[Theorem 9]{Gfact}) of Freudenburg gives the structure of the $n$-th degree module $\mathscr{F}_{n}$.

\begin{thm}\thlabel{ft}
	Let $s$ be a non-negative integer. Then with respect to the above notation, the following conditions are equivalent:
	\begin{enumerate}
		\item[\rm(a)]  $fB \cap M_{s}=fM_{s}$.
		\item[\rm(b)] $\mathscr{F}_{n}=M_{s}$.
	\end{enumerate}
\end{thm}

We now recall the definition of a $D$-set and a $D$-basis (\cite[Definition 1]{Gfact}).
\begin{definition}\thlabel{dbasis}
	\em{A subset $S$ of $B$ is said to be a {\it $D$-set}, if $\deg_{D}$ values of the elements of $S$ are distinct. 
		Let $F$ be a free $A$-submodule of $B$. A basis for $F$ is said to a {\it $D$-basis}, if that is a $D$-set.}
\end{definition}

The following lemma (\cite[Lemma 3]{Gfact}) by Freudenburg gives a condition for freeness of an $A$-submodule of $B=k[X,Y,Z]$.

\begin{lem}\thlabel{fl}
	Let $M$ be an $A$-submodule of $B$ generated by $\{m_{i}\mid 1 \leqslant i \leqslant n\} \subset M$. Suppose there exists $h \in B$ such that $\deg_{D}(m_{i}) < \deg_{D}(h)$ for  $1 \leqslant i \leqslant n$.
	Then, for the $A$-module $M^{\prime}= \sum_{i \geqslant 0}Mh^i$, the following properties hold.
	\begin{enumerate}
		\item [\rm(a)] $M^{\prime}= \bigoplus_{i \geqslant 0}Mh^i$
		\item[\rm(b)] If $M$ is a free $A$-module with $D$-basis $\{b_{1},\cdots,b_{n}\}$, then $M^{\prime}$ is a free $A$-module with a $D$-basis of the form $\{b_{i}h^j\,\,| 1 \leqslant i \leqslant n, j \geqslant 0\}$
	\end{enumerate} 
\end{lem}

%

We now state the Quillen-Suslin theorem. 
\begin{thm}\thlabel{qs}
	Every finitely generated projective module over $k^{[n]}$ is free.
\end{thm}

\subsection{Generators of the image ideals}

We first observe the following lemmas. The first one is a generalised version of Corollary 15 of \cite{Gfact}. 

\begin{lem}\thlabel{fc}
	Let $D$ be a locally nilpotent derivation on $B=k[X,Y,Z]$, $A=ker(D)$ and $\deg_{D}(Z)=n>0$. Consider the following surjective $A$-module morphism
	$$\pi : B \rightarrow \frac{B}{ZB}.$$ Suppose there exists a degree module $\mathscr{F}_{m}$ such that $m <n$ and $\pi(\mathscr{F}_{m})=k[X,Y]$. Then $B= \sum_{i \geqslant 0}^{} \mathscr{F}_{m}Z^i$.
\end{lem}
\begin{proof}
	Since $\pi(B)=\pi(\mathscr{F}_{m})$, $B=\mathscr{F}_{m}+ZB$. That means,
	$$
	B= \sum_{i=0}^{r-1}\mathscr{F}_{m}Z^i+Z^rB,
	$$ for every $r \geqslant 1$.
	Therefore, for $N=\sum_{i \geqslant 0}^{} \mathscr{F}_{m}Z^i$, we get $B=N+Z^{r}B$ for all $r \geqslant 1$.\\Let $b\in B$ and $t=\deg_{D}(b)$. Since $B=N+Z^{t+1}B$, if $b \notin N$, then there exist 
	$a=\sum_{i=0}^{t}a_{i}Z^i \in N$ and 
	$b^{\prime} \in B$ such that $a_{i}\in \mathscr{F}_{m}$ for 
	$1 \leqslant i \leqslant t$, $b^{\prime} \neq 0$ and 
	$$
	b=a+ b^{\prime}Z^{t+1}.
	$$
	Now $\deg_{D}(a) \leqslant tn+m$ and $\deg_{D}(b^{\prime}Z^{t+1}) \geqslant n(t+1)$. Since $m <n$, we get $\deg_{D}(b)= \deg_{D}(b^{\prime}Z^{t+1}) \geqslant n(t+1)$. But this contradicts the assumption that $\deg_{D}(b)=t$. Therefore, $b \in N$. Hence, $B=N$.
\end{proof}

For the next lemma one may refer to \cite[pg. 5]{Dfree}. However, for the sake of completeness of this note we are giving a proof here.
\begin{lem}\thlabel{ac}
	Let $B=k[X,Y,Z]$ and $D \in \LND(B)$. 
	Let $\overline{k}$ be an algebraic closure of 
	$k$ and $\overline{D}=D \otimes_{k} \overline{k}$ 
	denote the extension of $D$ in $\LND(\overline{B})$, where $\overline{B}=\overline{k}[X,Y,Z]$. For $n\in \mathbb{N}$, suppose $I_n$ and $\overline{I_{n}}$ denote the $n$-th image ideal of $D$ and $\overline{D}$ respectively. If $\overline{I_n}$ is principal, then so is $I_n$.  
	
\end{lem}

\begin{proof}
	Let $A=ker(D)=k^{[2]}$ and $\overline{A}=A \otimes_{k} \overline{k}=ker(\overline{D})=\overline{k}^{[2]}$. Now, $I_{n}=D^nB \cap A$ and $\overline{I_{n}}=\overline{D}^n \overline{B} \cap \overline{A}$, for every integer $n \geqslant 0$. 
	
Suppose that $\overline{I_{n}}=(\overline{a_{n}})$ for some $\overline{a_{n}} \in \overline{A}$. That means $\overline{I_{n}}$ is a free $\overline{A}$-module. Since $\overline{I_n}=I_n \otimes_{k} \overline{k}=I_n \otimes_{A} \overline{A}$, and $\overline{A}$ is a faithfully flat $A$-module, we have $I_n$ is a projective $A$-module. Hence by \thref{qs}, we have $I_n$ is free $A$-module, and hence a principal ideal.
\end{proof}

Let $D$ be a homogeneous triangularizable LND on $k^{[3]}$. Then, $D=a D^{\prime}$ for some $a \in ker(D)$ and an irreducible triangularizable LND $D^{\prime}$. For $n \in \mathbb{Z}$, if $I_n$ and $I_n^{\prime}$ denote the $n$-th image ideals of $D$ and $D^{\prime}$ respectively, then $I_n= a I_n^{\prime}$. Hence it is enough to find the generators of the image ideals of irreducible homogeneous triangularizable LNDs.  

The following theorem explicitly describes the image ideals of irreducible homogeneous triangularizable LNDs  on $k^{[3]}$. 

\begin{thm}\thlabel{trf}
	
	Let $B=k[U,V,W]$ and $D \in \LND(B)$ be irreducible homogeneous and triangularizable of degree $d (\geqslant 0)$. Let $A=ker(D)$. Then for $n \in \mathbb{N}$, we have $I_{n}=(X^{t(d+1)^2+r(d+1)})$ where  $t$ and $r$ are respectively the quotient and reminder of $n$ when divided by $d+2$, and $X$ is a linear variable of $B$ such that $X \in ker(D)$.
	
\end{thm}

\begin{proof}
	Since $D$ is irreducible and triangularizable, 
	by \thref{tr}, we get a system of variables $\{X,Y,Z\}$ which are linear in $\{U,V,W\}$ such that $D=\gamma \Delta_{(X,P)}$ where 
	$$
	P=Y^{d+2}+Xf_{d+1}(X,Y)+ \beta X^{d+1}Z,
	$$
	 $\gamma ,\beta \in k^{*}$, $f_{d+1}(X,Y)$ is homogeneous polynomial of degree $d+1$ and  $\deg_{D}(Y)=1$, $\deg_{D}(Z)=d+2$. Now $A=ker(D)=k[X,P]$. 
	 Consider the following surjective $A$-algebra homomorphism 
	$$
	\pi : B \rightarrow \frac{B}{ZB}.
	$$
	Now $\frac{B}{ZB} = k[X,Y]$ and $\pi(A)=k[X, Y^{d+2}+Xf_{d+1}(X,Y)]$. Clearly $k[X,Y]$ is free $\pi(A)$-module with a basis 
	$$
	 \mathscr{A}= \{1,Y,Y^2,\ldots,Y^{d+1}\}.
	 $$
	  As $\deg_D(Y)=1$, $\mathscr{A} \subset \pi(\mathscr{F}_{d+1}) $.
	Hence we get $\pi(\mathscr{F}_{d+1})=k[X,Y]$ as $\pi(A)$-module. Now by \thref{fc}, we obtain that
	$$
	B= \sum_{i \geqslant 0}^{} \mathscr{F}_{d+1}Z^{i}.
	$$
	Since $\deg_{D}(Z)=d+2$, by \thref{fl}(a), we get the above sum is direct sum. That is 
	$$
	B= \bigoplus_{i \geqslant 0} \mathscr{F}_{d+1} Z^{i}.
	$$
	Now if we show $\mathscr{F}_{d+1}$ is a free $A$-module having a $D$-basis, then applying \thref{fl}(b), we get a $D$-basis for $B$.
	
	 Consider the $A$-submodule $M_{0}=\bigoplus_{i=0}^{d+1} A Y^{i}$ of $\mathscr{F}_{d+1}$. Note that $DY= -\gamma \beta X^{d+1}$. Now consider the surjection $$
	\pi^{\prime}: B \rightarrow \frac{B}{XB}.
	$$ 
	Since $\pi^{\prime}(A)=k[Y^{d+2}]$, $\pi^{\prime}(M_{0})=\bigoplus_{i=0}^{d+1}$ $\pi^{\prime}(A)Y^{i} $ is a free $\pi^{\prime}(A)$-module. Therefore, $XB \cap M_{0} = XM_{0}$ and hence $(X^{d+1})B \cap M_{0}= (X^{d+1})M_{0}$.
	
	 Now applying \thref{ft}, we get that $\mathscr{F}_{d+1}=M_{0}$, which is a free $A$ module with $D$-basis $\{1, Y,\ldots, Y^{d+1}\}$. Therefore, we obtain that $B$ is free $A$-module with a $D$-basis $\{Y^{i}Z^{j} \mid 0 \leqslant i \leqslant d+1, j \geqslant 0\}$.

	 Let $n$ be a positive integer such that $n=t(d+2)+r$, for some $t \geqslant 0$ and $0 \leqslant r \leqslant d+1$. Now $\deg_{D}(Y^rZ^t)=r+t(d+2)$. Therefore, from the $D$-basis it is clear that $I_{n}=\left(D^n(Y^rZ^t)\right)$. Note that $D$ is a homogeneous LND of degree $d$, $DY= -\gamma\beta X^{d+1}$ and $DZ= \gamma P_Y$. Therefore, $D^{d+2}Z= \lambda X^{d(d+2)+1}$ for some $\lambda \in k^{*}$ and hence $D^n(Y^rZ^t)$ is a constant multiple of a power of $X$. Now using the homogeneous degree of $D$, we have $I_{n}=(X^{nd+t+r})=(X^{t(d+1)^2+r(d+1)})$.
\end{proof}

We now describe the generators of the image ideals of irreducible homogeneous non-triangularizable LNDs of rank $2$ and degree $pq-2$ on $k^{[3]}$ where $p,q$ are prime numbers, not necessarily distinct.

\begin{thm}\thlabel{ntrf}
	Let $p,q$ be prime numbers, not necessarily distinct and $D$ be an irreducible non-triangularizable homogeneous locally nilpotent derivation of rank $2$ and degree $pq-2$ on $B=k[U,V,W]$ and $A=ker(D)$. Let $X$ be a linear variable of $B$ such that $X \in ker(D)$. Then the image ideals have either of the following forms:
	
	\begin{itemize}
		\item [\rm(a)] For every $n \in \mathbb{N}$, $I_{n}=(X^{n(pq-2)+qr+s+t})$, where $r$ is the reminder  of $n$ modulo $p$, and if $t^{\prime}$ is the quotient of $n$ when divided by $p$, then $t$ and $s$ are the quotient and reminder of $t^{\prime}$ when divided by $q$.
		
		\item[\rm(b)] For every $n \in \mathbb{N}$, $I_{n}=(X^{n(pq-2)+pr+s+t})$, where $r$ is the reminder  of $n$ modulo $q$, and if $t^{\prime}$ is the quotient of $n$ when divided by $q$, then $t$ and $s$ are the quotient and reminder of $t^{\prime}$ when divided by $p$.
	\end{itemize}
	

%
\end{thm}
\begin{proof}
	
	Let $\overline{k}, \overline{B}, \overline{D}, \overline{I_{n}}$ be the same as in \thref{ac}.
	 As $D$ is irreducible, so is $\overline{D}$ (cf. \thref{irr}). Also $\deg(D)=\deg(\overline{D})$. As $D$ is non-triangularizable,
	 by \thref{rk1},
	  we get $\overline{D}$ is also non-triangularizable. 
	   By \thref{ac}, $\overline{I_{n}}$ is principal if and only if $I_n$ is so. Now, since $\overline{I_n}=I_n \otimes_{k} \overline{k}$, if we assume $\overline{I_n}$ is principal, then the generator of $\overline{I_{n}}$ is same as the generator of $I_n$ upto multiplication by a non-zero constant in $\overline{k}$.
	  Therefore, it is enough to assume that $k$ is algebraically closed, $\overline{B}=B$ and $\overline{D}=D$.
	
	
	Now, by \thref{ntr} there exists a system of variables $\{X,Y,Z\}$ of $B$ which are linear in $\{U,V,W\}$ such that 
	$D=\gamma \Delta_{(X,P)}$ for some $\gamma \in k^{*}$, and $P$ takes either of the following forms:
	
	\begin{equation}\label{p}
			P=T^{p}+ c_{1}X^{q}T^{p-1}+\cdots+c_iX^{iq}T^{p-i}+\dots+
			c_{p-1}X^{pq-q}T+  
		c_{p}X^{pq-1}Y,
	\end{equation}
where $T=h(X,Y)+  X^{q-1}Z$, $h(X,Y)$ is a homogeneous polynomial of degree $q$ and monic in $Y$, $c_{i} \in k$, $c_{p} \neq 0$ and $\deg_{D}(Y)=p, \deg_{D}(Z)=pq$. Or, 

\begin{equation}\label{q1}
	P=T_1^{q}+ \widetilde{c}_{1}X^{p}T_1^{q-1}+\cdots+\widetilde{c}_iX^{ip}T_1^{q-i}+\dots+
	\widetilde{c}_{q-1}X^{pq-p}T_1+  
	\widetilde{c}_{q}X^{pq-1}Y,
\end{equation}
where $T_1=h_1(X,Y)+  X^{p-1}Z$, $h_1(X,Y)$ is a homogeneous polynomial of degree $p$ and monic in $Y$, $\widetilde{c}_{i} \in k$, $\widetilde{c}_{q} \neq 0$ and $\deg_{D}(Y)=q, \deg_{D}(Z)=pq$.
Therefore, without loss of generality we take $P$ as in \eqref{p} and proceed.

	Let $A=ker(D)$. 
	Now, as in \thref{ntr}, $T$ is a local slice of $D$ such that $DT=\gamma c_{p}X^{pq+q-2}$.  
	Consider the surjective $A$-algebra homomorphism 
	$$
	\pi: B \rightarrow \frac{B}{ZB}.
	$$
	Note that $\pi(P)=Y^{pq}+a_{1}XY^{pq-1}+\cdots+a_{pq-1}X^{pq-1}Y$, where $a_{i} \in k$, for $1 \leqslant i \leqslant pq-1$.
	Now $\frac{B}{ZB}=k[X,Y]$ and $\pi(A)= k[X,Y^{pq}+a_{1}XY^{pq-1}+\cdots+a_{pq-1}X^{pq-1}Y]$. Therefore, $k[X,Y]$ is a free $\pi(A)$-module with basis
	$$
	\mathscr{B}=\{1,Y,\ldots,Y^{pq-1}\}.
	$$
	Now $S :=\{1,Y,\ldots,Y^{q-1}, T, TY,\ldots,TY^{q-1},\ldots, T^{p-1}, T^{p-1}Y,\ldots,T^{p-1}Y^{q-1}\} \subset \mathscr{F}_{pq-1}$. Since $\pi(S) \subset \pi(\mathscr{F}_{pq-1})$ and $\pi(T)$ is monic in $Y$ of degree $q$, it follows that $\mathscr{B} \subset \pi(\mathscr{F}_{pq-1})$.
	Hence we obtain that $\pi(\mathscr{F}_{pq-1})=k[X,Y]$ as $\pi(A)$-modules. Therefore, by \thref{fc} we get 
	$
	B=\sum_{i \geqslant 0} \mathscr{F}_{pq-1}Z^i.
	$
	Since $\deg_{D}(Z)=pq$, by \thref{fl}(a) we get 
	$$
	B=\bigoplus_{i \geqslant 0}\mathscr{F}_{pq-1}Z^i.
	$$
	We now show that $\mathscr{F}_{pq-1}$ is free $A$-module with a $D$-basis. Take  the $A$-submodule of $\mathscr{F}_{pq-1} $ as follows:
	$$
	N_{0}=\bigoplus_{\substack{0 \leqslant i \leqslant q-1 \\
			0 \leqslant j \leqslant p-1}} AY^{i}T^{j}.
	$$
	Note that from \eqref{p}, $$
	T^{p}=P-c_{1}X^{q}T^{p-1}-\cdots-c_iX^{iq}T^{p-i}-\dots-
	c_{p-1}X^{pq-q}T-  
	c_{p}X^{pq-1}Y.
	$$
	Therefore, $N:= \sum_{i=0}^{pq-1} A T^i \subseteq N_0$.
	Now consider the surjection 
	$$
	\pi^{\prime}: B \rightarrow \frac{B}{XB}.
	$$ 
	$\pi^{\prime}(A)=k[Y^{pq}]$ and $\pi^{\prime}(N_{0})=\bigoplus_{i=0}^{pq-1} \pi^{\prime}(A)Y^i$ is a free $\pi^{\prime}(A)$-module. Therefore,
	 $XB \cap~N_{0}=XN_{0}$ and hence $(X^{pq+q-2})B \cap N_{0}= (X^{pq+q-2})N_{0}$ i.e., $(DT) B \cap N_{0}= (DT)N_{0}$. 
	By \thref{ft}, $\mathscr{F}_{pq-1}=N_{0}$ is a free $A$-module with $D$-basis 
	$$
	\{Y^{i}T^{j} \mid 0 \leqslant i \leqslant q-1, 0 \leqslant j \leqslant p-1 \}.
	$$ 
	Hence $B$ is a free $A$-module with $D$-basis
	$$
	\mathscr{B}_{1}=\{Y^iT^jZ^l \,\, \mid 0 \leqslant i \leqslant q-1, 0 \leqslant j \leqslant p-1  , l \geqslant 0\}.
	$$
	
	\medskip
	 Let $n=t^{\prime}p+r$ for some $t^{\prime} \geqslant 0$ and $0 \leqslant r \leqslant p-1$, and $(t^{\prime},q)$ we have $t^{\prime}=tq+s$ 
	 for some $t \geqslant 0$ and $0 \leqslant s \leqslant q-1$. Therefore, $n=tpq+sp+r$. 
	 Now $\deg_D(Y^sT^rZ^t)=sp+r+tpq=n$. Therefore,
	 from the structure of the $D$-basis it is clear that $I_{n}=\left( D^n(Y^sT^rZ^t) \right)$. As $DT=\gamma c_{p}X^{pq+q-2}$, $DY=-\gamma P_Z$ and $DZ=\gamma P_Y$, from \eqref{p}, it is clear that $D^pY$ and $D^{pq}Z$ are constant multiples of some powers of $X$. Now, since $\deg(D)=pq-2$ and $T$ is a polynomial of degree $q$, we get $I_{n}=(X^{n(pq-2)+qr+s+t})$ which is same as the structure in (a).
	 
	 If the structure of $P$ is as in \eqref{q1}, proceeding similarly we get the generators of the image ideals as in (b).
\end{proof}

\section*{Acknowledgement}

The author is thankful to Prosenjit Das and Neena Gupta for carefully going through the earlier draft and suggesting several improvements. 
The author acknowledges the Council of Scientific and Industrial Research (CSIR) for the Shyamaprasad Mukherjee Fellowship (File No. SPM-07/0093(0295)/2019-EMR-I).

\end{document}